\documentclass[11pt]{article}

\usepackage[a4paper, margin=1in]{geometry}
\usepackage{amsfonts,amsmath,amssymb,amsthm,graphicx,fixmath,latexsym,color}
\usepackage{xcolor}
\usepackage{hyperref,epsfig}
\usepackage{algorithmic}
\usepackage{algorithm}
\usepackage{xspace}

\providecommand{\remove}[1]{}

\renewcommand{\Re}{\mathbb{R}}

\newcommand{\C}{\mathcal{C}}

\newcommand{\F}{\mathcal{F}}

\newtheorem{theorem}{Theorem}[section]
\newtheorem{lemma}[theorem]{Lemma}
\newtheorem{proposition}[theorem]{Proposition}
\newtheorem{claim}[theorem]{Claim}
\newtheorem{corollary}[theorem]{Corollary}

\newtheorem{definition}[theorem]{Definition}
\newtheorem{remark}[theorem]{Remark}

\newtheorem*{theorem*}{Theorem}
\newtheorem*{lemma*}{Lemma}
\newtheorem*{proposition*}{Proposition}

\newtheorem*{example*}{Example}

\begin{document}
\title{Conflict-Free Coloring of Intersection Graphs of Geometric Objects}

\author{Chaya Keller\thanks{Department of Mathematics, Ben-Gurion University of the NEGEV, Be'er-Sheva Israel. \texttt{kellerc@math.bgu.ac.il}. Research partially supported by Grant 635/16 from the Israel Science Foundation, the Shulamit Aloni Post-Doctoral Fellowship of the Israeli Ministry of Science and Technology, and by the Kreitman Foundation Post-Doctoral Fellowship.}
\and
Shakhar Smorodinsky\thanks{Department of Mathematics, Ben-Gurion University of the NEGEV, Be'er-Sheva Israel.
\texttt{shakhar@math.bgu.ac.il}. Research partially supported by Grant 635/16 from the Israel Science Foundation.}
}

\date{}
\maketitle

\begin{abstract}
In FOCS'2002, Even et al. introduced and studied the notion of conflict-free colorings of geometrically defined hypergraphs.  They motivated it by frequency assignment problems in cellular networks. This notion has been extensively studied since then.

A conflict-free coloring of a graph is a coloring of its vertices such that the neighborhood (pointed or closed) of each vertex contains a vertex whose color differs from the colors of all other vertices in that neighborhood.
In this paper we study conflict-colorings of intersection graphs of geometric objects. We show that any intersection graph of $n$ pseudo-discs in the plane admits a conflict-free coloring with $O(\log n)$ colors, with respect to both closed and pointed neighborhoods. We also show that the latter bound is asymptotically sharp.
Using our methods, we also obtain a strengthening of the two main results of Even et al. which we believe is of independent interest.
In particular, in view of the original motivation to study such colorings, this strengthening suggests further applications to frequency assignment in wireless networks.

Finally, we present bounds on the number of colors needed for conflict-free colorings of other classes of intersection graphs, including intersection graphs of axis-parallel rectangles and of $\rho$-fat objects in the plane.
\end{abstract}

\section{Introduction}
\label{sec:introduction}

\subsection{Background}
A {\em conflict-free coloring} ({\em CF-coloring} in short) of a hypergraph $H$ is a coloring of its vertices in which each hyperedge contains a uniquely colored vertex (i.e., a vertex whose color differs from the colors of all other vertices that belong to the hyperedge).

Conflict-free colorings of hypergraphs were introduced in FOCS'2002 by Even et al.~\cite{ELRS}, in the context of the following problem of frequency assignment in cellular networks.
Suppose we need to assign radio frequencies to a given set of {\em base stations}  (i.e., antennas) in a cellular network.
Assume that the location and the transmission radius of each base station in the network is known and fixed, and modeled as a disc in the plane. Any client continuously scans frequencies in search of a base station with a good reception. If a client is within the reception range of two base stations which have the same frequency then mutual interference occurs. The goal is to allocate the minimum number of frequencies to the base stations such that any client (modeled as a point in the plane) will be able to connect to the network without interference (i.e., that among the antennas whose signals the client receives there will always be an antenna whose frequency differs from all other antennas).

The problem corresponds to conflict-free coloring of the hypergraph whose
vertex set is $B=\{D_1,D_2,\ldots,D_n\}$ -- the set of discs that correspond to the base stations, and whose hyperedges are all sets of the
form $E_x =\{D_i \in B: x \in D_i\}$, for all $x \in \Re^2$. The first main result proved in~\cite{ELRS} asserts that this hypergraph admits a CF-coloring with $\Theta(\log n)$ colors in the worst case:
\begin{theorem}[\cite{ELRS}]\label{thm:ELRS1}
Any set $S=\{D_1,D_2,\ldots,D_n\}$ of $n$ discs in the plane can be colored with $O(\log n)$ colors such that for any $x \in \Re^2$, the set $E_x =\{D_i \in S: x \in D_i\}$ contains a uniquely-colored element. This bound is asymptotically tight.
\end{theorem}
Another main result of~\cite{ELRS} is the following `dual' theorem:
\begin{theorem}[\cite{ELRS}]\label{thm:ELRS2}
Any set $S=\{x_1,x_2,\ldots,x_n\}$ of $n$ points in the plane can be colored with $O(\log n)$ colors such that for any disc $D$ in the plane, the set $E_D =\{x_i \in S: x \in D\}$ contains a uniquely-colored vertex. This bound is asymptotically tight.
\end{theorem}
In fact, the assertion of Theorem~\ref{thm:ELRS2} holds in the more general setting of scaled translates of any fixed convex body instead of discs. The tightness assertion also holds in a stronger form (specifically, for \emph{any} set of $n$ points), as shown by Pach and T\'{o}th~\cite{cf1}.

The work of Even et al. was the starting point of an entire field of research in computational geometry. Dozens of papers studied conflict-free colorings of graphs and hypergraphs, mostly arising from geometric objects. See the survey \cite{CF-survey} and the references therein for more on CF-colorings.

For a simple undirected graph $G=(V,E)$ and a vertex $v \in V$, the \emph{pointed neighborhood} of $v$ in $G$
is $N_G(v)=\{u \in V: (v,u) \in E\}$.
The \emph{closed neighborhood} of $v$ is $N_G[v] = N_G(v) \cup \{v\}$. A coloring of the vertices of $G$ is
called \emph{conflict-free} with respect to closed (respectively, pointed) neighborhoods,
or in short, \emph{closed CF-coloring} (respectively, \emph{pointed CF-coloring}) if for any $v \in G$,
the set $N_G[v]$ (respectively, $N_G(v)$) contains a unique color (namely,
if there exists $u$ in the closed (respectively, pointed) neighborhood of $v$ such
that the color of $u$ differs from the color of all other vertices in the neighborhood).
The minimum number of colors in a CF-coloring of $G$ with respect to closed (respectively, pointed) neighborhoods
is denoted by $\chi_{CF}^{cn}(G)$ (respectively, $\chi_{CF}^{pn}(G)$).
Note that a closed (respectively, pointed) CF-coloring of a graph $G=(V,E)$ is
a CF-coloring of the hypergraph on the same vertex set whose hyperedges are all sets of the form $N_G[v]$
(respectively, $N_G(v)$), for all $v \in V$.
In the case of pointed/closed CF-coloring of \emph{graphs}, several papers considered the abstract setting (i.e., not assuming any geometric structure on the graph). In particular, tight upper bounds on $\chi_{CF}^{pn}(G)$ and $\chi_{CF}^{cn}(G)$ for general graphs are known, and they suggest that generally, the problem of pointed CF-coloring is harder than the problem of closed CF-coloring.

On the one hand, for any graph $G$ we have $\chi_{CF}^{cn}(G) \leq 2\chi_{CF}^{pn}(G)$ (see~\cite{CFPT09}). Indeed, any pointed CF-coloring is already `almost' a closed CF-coloring -- the only vertices that may not have a unique-colored neighbor are those whose degree in the induced subgraph created by their color class is 1. These vertices can be easily taken care of by recoloring each color class with two colors such that the color of each leaf will differ from the color of its only neighbor. Hence, at most twice as many colors are sufficient for a closed CF-coloring.

On the other hand, there exist graphs for which $\chi_{CF}^{cn}(G)$ is constant, while $\chi_{CF}^{pn}(G)$ is arbitrarily large. For example, let $G$ be the $n$-vertex graph obtained by taking a 1-subdivision of the complete graph $K_i$ (i.e., replacing each edge $(v,w)$ with the two edges $(v,u),(u,w)$ where $u$ is a new vertex). As $G$ is bipartite, we have $\chi_{CF}^{cn}(G)=2$. (Indeed, any proper coloring of a graph is a closed CF-coloring, as any vertex itself is a unique-colored vertex in its neighborhood.) On the other hand, any pair of vertices of the original graph composes a pointed neighborhood of a new vertex, and thus any pointed CF-coloring requires at least $i$ colors. Since the total number of vertices is $n=i+\binom{i}{2}$, we have $\chi_{CF}^{pn}(G)=\Omega(\sqrt{n})$.

Cheilaris~\cite{chei09} showed that for any $n$-vertex graph $G$, we have $\chi_{CF}^{pn}(G) \leq 2\sqrt{n}$. Pach and Tardos~\cite{CFPT09} noted that the constant can be improved to $\sqrt{2}$, and in fact, one can show that any hypergraph with less than $\binom{m}{2}$ hyperedges can be CF-colored with less than $m$ colors, which is almost tight in view of the above example. The upper bound for closed CF-coloring is much smaller: Pach and Tardos~\cite{CFPT09} showed that for any $G$, $\chi_{CF}^{cn}(G) = O(\log^2 n)$, and Glebov, Sz\'{a}bo, and Tardos~\cite{GST14} showed that this bound is asymptotically tight. A recent result of Abel et al.~\cite{AADFGHKS17} shows that a slightly modified variant of the classical Hadwiger conjecture holds for closed CF-colorings: any graph $G$ that does not contain $K_{r+1}$ as a minor admits a `closed CF-coloring' with $r$ colors.\footnote{In fact, in~\cite{AADFGHKS17} it is allowed to leave vertices uncolored; in our model, an additional color may be required for all uncolored vertices, and thus, in our notation the bound is $\chi_{CF}^{cn}(G) \leq r+1$.}

Another recent result, of Fox, Pach and Suk~\cite{FPS17}, exhibits a strong relation between properties of the hypergraph $H$ whose hyperedges are pointed neighborhoods of vertices of an $n$-vertex graph $G$ (which is our main object of study) to properties of the graph $G$. Specifically, it is shown in~\cite{FPS17} that if the VC-dimension of $H$ is bounded (which is the case in intersection graphs of various geometric objects) (see, e.g., \cite{MATOUSEK}) then $G$ admits a stronger Ramsey theorem which assures existence of a clique or an independent set of size $r$ with $r=\exp(\log^{1-o(1)}(n))$ (instead of $r \approx \frac{1}{2}\log n$ assured by the classical Ramsey theorem).

\subsection{Results}

In this paper we study pointed and closed CF-colorings of graphs from a geometric viewpoint. We consider CF-coloring of intersection graphs of geometric objects in the plane, such as pseudo-discs and axis-parallel rectangles. The vertices of such a graph are the geometric objects, and two vertices are connected by an edge if the intersection of the corresponding objects is non-empty. We show that in this setting, the bounds on the CF-chromatic number are much smaller than the (tight) general upper bounds mentioned above.

Our main result concerns intersection graphs of families of pseudo-discs in the plane -- i.e., families of simple closed Jordan regions such that the boundaries of any two regions intersect in at most two points.
\begin{theorem}\label{Thm:Pseudo-discs}
Let $\F$ be a family of $n$ pseudo-discs in the plane and let $G$ be its intersection graph. Then $\chi_{CF}^{pn}(G) = O(\log n)$, and this bound is asymptotically tight.
\end{theorem}
Note that by the argument given above, Theorem~\ref{Thm:Pseudo-discs} also implies that $\chi_{CF}^{cn}(G) = O(\log n)$. However, we do not know whether this bound is tight.

Using the techniques we develop in the course of the proof of Theorem~\ref{Thm:Pseudo-discs}, we provide a strong generalization of the first major results on CF-coloring -- Theorems~\ref{thm:ELRS1} and~\ref{thm:ELRS2} presented above.
\begin{theorem}\label{Thm:Antennas}
Any family $\F$ of $n$ discs in the plane can be colored with $O(\log n)$ colors in such a way that for any disc $B$ (not necessarily from $\F$), the set $S_B = \{D \in \F: D \cap B \neq \emptyset\}$ contains a uniquely-colored element. The bound $O(\log n)$ is asymptotically tight.
\end{theorem}
It is easy to verify that Theorems~\ref{thm:ELRS1} and~\ref{thm:ELRS2} are special cases of Theorem~\ref{Thm:Antennas} in which some of the discs have radius equal to zero, namely these are points.

Theorem~\ref{Thm:Antennas} generalizes the original application of CF-coloring (for frequency assignment) to the scenario in which each client is a receiver that also has a range of reception, and thus, should be modeled by a disc rather than by a single point. Thus $O(\log n)$ frequencies are sufficient to assign to the given $n$ antennas no matter where the mobile clients are placed and no matter what their reception range is.

In addition, we show that Theorem~\ref{Thm:Antennas} holds in the stronger setting of \emph{list coloring}, in which the color of each disc must be chosen from a set of $O(\log n)$ prescribed colors (which corresponds to the case where each antenna is equipped with a list of $O(\log n)$ possible frequencies):
\begin{theorem}\label{Thm:List}
Let $\F$ be a family of $n$ discs in the plane, and assume that each element $D \in \F$ is assigned a list $S_D$ of $O(\log n)$ possible colors (where there is no assumption on the relation between the lists of different vertices). Then one can color the elements of $\F$ by choosing, for each disc, one of the colors from its list in such a way that for any disc $B$, the set $S_B = \{D \in \F: D \cap B \neq \emptyset\}$ contains a uniquely-colored element. The bound $O(\log n)$ (on the size of the lists) is asymptotically tight.
\end{theorem}

{\noindent \bf Intersection graphs of other geometric objects}

\medskip

For intersection graphs of axis-parallel rectangles in the plane, we obtain an upper bound on the closed CF-chromatic number.
\begin{theorem}\label{Thm:Rectangles}
Let $\F$ be family of $n$ axis-parallel rectangles in the plane and let $G$ be its intersection graph. Then $\chi_{CF}^{cn}(G) = O(\log n)$.
\end{theorem}
We do not know whether the bound is tight, and whether a similar bound can be proved in the harder case of pointed CF-coloring.


%

\medskip

Finally, we provide an upper bound on the pointed and closed chromatic numbers of intersection graphs of $\rho$-fat objects (i.e., simple Jordan regions $C$ in the plane for which there exists $x \in C$ and discs $A,B$ centered at $x$ with radiuses $r_A,r_B$ (respectively), such that $A \subset C \subset B$ and $r_B/r_A \leq \rho$).
\begin{proposition}\label{Thm:rho-fat}
For any $\rho,k \geq 1$, let $\F$ be a finite family of $\rho$-fat objects in the plane such that the size-ratio of $\F$ is $k$, and let $G$ be its intersection graph. Then $\chi_{CF}^{pn}(G) = O(\rho^2 k^2)$ and $\chi_{CF}^{pn}(G) = O(\rho^2 \log k)$.
\end{proposition}
An immediate corollary of Proposition~\ref{Thm:rho-fat} is that if $\F$ is a family of translates of any simple Jordan region with a non-empty interior, then its intersection graph admits a pointed CF-coloring with a constant number of colors.

\subsection{Organization of the paper}

The rest of the paper is organized as follows. In Section~\ref{sec:pd} we study CF-coloring of the intersection graph of pseudo-discs, and prove Theorems~\ref{Thm:Pseudo-discs},~\ref{Thm:Antennas}, and~\ref{Thm:List} (except for the tightness of Theorem~\ref{Thm:Pseudo-discs} that is proved in Appendix~\ref{App:Tightness}). In Appendix~\ref{sec:rectangles} we study the case of axis-parallel rectangles and prove Theorem~\ref{Thm:Rectangles}. We discuss $\rho$-fat objects and prove Proposition~\ref{Thm:rho-fat} in Appendix~\ref{sec:rho-fat}. Finally, we present an alternative proof of (part of) Theorem~\ref{Thm:Pseudo-discs} in Appendix~\ref{sec:alternative}.

\section{Intersection Graph of Pseudo-discs}
\label{sec:pd}

In this section we show that the intersection graph of a family of $n$ pseudo-discs admits a pointed CF-coloring with $O(\log n)$ colors. Our technique also allows us to prove Theorem~\ref{Thm:Antennas} which we believe is of special interest. Indeed, as mentioned in the introduction, the two main results of~\cite{ELRS} follow immediately from our Theorem~\ref{Thm:Antennas} as special cases.
We start with a few definitions and state previous results that we use in the course of our proofs.

\subsection{Definitions}

A hypergraph is a pair $H=(V,E)$, where $E$ (the hyperedge set) is a set of subsets of $V$. For a subset $V' \subset V$ the induced sub-hypergraph $H(V')$ is the hypergraph $(V',E')$, where $E'=\{e \cap V': e \in E\}$.

A proper coloring of a hypergraph with $k$ colors is a function $\varphi:V \rightarrow \{1,2,\ldots,k\}$ such that any $e \in E$ with $|e| \geq 2$ is not monochromatic, meaning that there exist $x,y \in e$ with $\varphi(x) \neq \varphi(y)$. $\chi(H)$ is the least integer $k$ such that $H$ admits a proper coloring with $k$ colors.

A conflict-free (CF) coloring of a hypergraph $H$ with $k$ colors is a function $\psi: V \rightarrow \{1,2,\ldots,k\}$ such that for any $e \in E$ there exists  $x \in e$ that is uniquely colored,
meaning that $\psi(y) \neq \psi(x)$ for all $x \neq y \in e$. $\chi_{CF}(H)$ is the least integer $k$ such that $H$ admits a CF-coloring with $k$ colors.

We will use the following theorem (proved in~\cite{smoro} in a much more general context), which relates proper (non-monochromatic) coloring of a hypergraph to a CF-coloring (though, with more colors).
\begin{theorem}[\cite{smoro}]\label{Lemma:Proper-to-CF}
Let $H$ be a hypergraph. If there exists a constant $k$ such that $\chi(H') \leq k$ for any induced sub-hypergraph $H'$ of $H$, then $\chi_{CF}(H) \leq 1+ \log_{1+\frac{1}{k-1}}n = O( \log n)$.
\end{theorem}
%
%

\subsection{Proof of Theorem~\ref{Thm:Pseudo-discs}}

In this subsection we prove Theorem~\ref{Thm:Pseudo-discs}, namely, that the intersection graph of any family of $n$ pseudo-discs in the plane satisfies $\chi_{CF}^{pn}(G) = O(\log n)$, and that this bound is asymptotically tight.

Throughout this section, $\F$ denotes a family of $n$ pseudo-discs in the plane, and $G=(V,E)$ denotes the intersection graph of $\F$, meaning that $V=\F$, and $(d_1,d_2) \in E$ if $d_1 \cap d_2 \neq \emptyset$.
 $B$ denotes a maximum-size independent set in $V(G)$ (i.e., a maximum-size set of pairwise disjoint pseudo-discs). Put $p =|B|$.

Our proof is composed of three parts:
\begin{enumerate}
\item We present a coloring of the vertices of $B$ with $O(\log p)$ colors such that each vertex in $V \setminus B$ admits a uniquely-colored neighbor. Note that this step already implies that $\chi_{CF}^{cn}(G) = O(\log p)$, since we can color all vertices of $V \setminus B$ with a single additional color, and then, with respect to close neighborhoods, each vertex $v \in B$ will be a uniquely-colored `neighbor' of itself.

\item We present a coloring of the vertices of $V \setminus B$ with $O(\log(n-p))$ colors such that each vertex in $B$ admits a uniquely-colored neighbor. Together with the first step, this proves that $\chi_{CF}^{pn}(G) = O(\log n)$.

\item We present an example of an intersection graph of $n$ discs that satisfies $\chi_{CF}^{pn}(G) = \Omega(\log n)$, thus proving the tightness of the theorem. This part is presented in Appendix~\ref{App:Tightness}.
\end{enumerate}
For ease of reading, we present each step in a separate subsection.

\subsubsection{Coloring the vertices of $B$}

Let $H=(V,E)$ be the hypergraph whose vertex set is $V=B$ and whose hyperedges are all subsets of the form $e_d = \{b \in B: b \cap d \neq \emptyset\}$, for all $d \in V \setminus B$. Since $B$ is a maximal independent set in $V$, each $d \in V \setminus B$ has a neighbor in $B$, and thus, a CF-coloring of $H$ provides any $d \in V \setminus B$ with a uniquely-colored neighbor in $N_G(d)$, as desired.

By Theorem~\ref{Lemma:Proper-to-CF}, in order to show that $H$ admits a CF-coloring with $O(\log p)$ colors, it is sufficient to show that there exists a constant $k$ such that any induced sub-hypergraph $H' \subset H$ admits a proper coloring with $k$ colors. We will show this (for $k=4$) in two steps:

\medskip \noindent (1). We enlarge $H$ to $\hat{H}$ by adding hyperedges in a certain way. Note that this step is somewhat counter-intuitive, as it makes our job strictly harder.

\medskip \noindent (2). We show that $\hat{H}$, as well as all of its induced sub-hypergraphs, are 4-colorable.

\medskip \noindent \textbf{Step 1: Enlarging $H$.} For this step, we use the following lemma proved by Pinchasi~\cite{Pin14}.
\begin{lemma}[~\cite{Pin14}]\label{Lemma:Rom1}
Let $\F$ be a family of pseudo-discs. Let $d \in \F$ and let $x \in d$ be any point. Then $d$ can be continuously shrunk to the point $x$ in such a way that at each moment, $\F$ continues to be a family of pseudo-discs.
\end{lemma}

To enlarge $H$, we add to $V \setminus B$ `dummy' pseudo-discs iteratively using the process described in Lemma~\ref{Lemma:Rom1}. Specifically, for each $d \in V \setminus B$ with $|e_d| > 2$, we choose a point $x_{d}$ and perform the process of shrinking $d$ to the point $x_{d}$. At any stage during the process, the number of elements of $e_d$ that intersect the `shrinking $d$' either remains unchanged or decreases by 1 (if it decreases by more than 1 at once, this can be fixed by a small perturbation of the shrinking process). Hence, the process induces a finite set of equivalence classes of `shrunk versions' of $d$, arranged by the set of elements of $e_d$ that intersect the `shrinking $d$' at that moment. We take a single representative from each such equivalence class in which the `shrinking $d$' intersects at least two elements of $e_d$. This can be done in such a way that the boundaries of the representatives we take are disjoint. (Again, a small perturbation of the shrinking process may be needed here.) We add these representatives to $V \setminus B$ as `dummy' pseudo-discs.

We repeat this process for each $d \in V \setminus B$ with $|e_d| > 2$ (note that this can be done in such a way that the family $V \setminus B$, along with the added dummy pseudo-discs, remains a family of pseudo-discs), and denote the resulting family by $\widehat{V \setminus B}$. We then define $\hat{H}$ by setting the vertices to be $V(\hat{H})=V(H)=B$ and the hyperedges to be all sets of the form $e_d = \{b \in B: b \cap d \neq \emptyset\}$, for all $d \in \widehat{V \setminus B}$. Of course, $\hat{H}$ contains $H$, and thus, it is sufficient to prove that any induced sub-hypergraph $H' \subset \hat{H}$ admits a proper coloring with $4$ colors.

\medskip \noindent \textbf{Step 2: CF-Coloring  $\hat{H}$ with $O(\log p)$ colors.} For this step, we use
the fact that after the enlarging step, each hyperedge of size $>2$ contains a sub-hyperedge of size 2.

\begin{definition}
A hypergraph $H$ is called a \emph{rank-$i$ hypergraph} if for any hyperedge $e \in E(H)$ with $|e|>i$ there exists a hyperedge $e' \subset e$ with $|e'|=i$. \end{definition}

\begin{definition}
Given a hypergraph $H$, the Delauney graph $D_H$ is defined by $V(D_H)=V(H)$ and $E(D_H)= \{e \in E(H): |e|=2\}$.
\end{definition}
The basic observation we use here is that for any rank-2 hypergraph $H$, a proper coloring of the graph $D_h$ is also a proper coloring of the hypergraph $H$. Indeed, since $H$ is rank-2, any hyperedge $e \in E(H)$ with $|e| \geq 2$ contains an edge of $D_H$, and thus, is not monochromatic in the coloring.

\medskip

It is easy to see that by the way we defined the enlarged hypergraph $\hat{H}$, both $\hat{H}$ and any induced sub-hypergraph $H' \subset \hat{H}$ are rank-2 hypergraphs.
Hence, it will be sufficient to show that for any sub-hypergraph $H' \subset \hat{H}$, the Delauney graph $D_{H'}$ is \emph{planar}, and hence, 4-colorable~\cite{AH77a,AH77b}. By Theorem~\ref{Lemma:Proper-to-CF} this will imply that $\hat{H}$ admits a CF-coloring with $O(\log p)$ colors, and thus, $H$ admits a CF-coloring with $O(\log p)$ colors, as asserted.

\medskip

The fact that for any $H' \subset \hat{H}$, the graph $D_{H'}$ is planar, follows at once from another lemma proved in~\cite{Pin14}:
\begin{lemma}\label{Lemma:Rom2}
Let $B$ be a family of pairwise disjoint closed connected sets in $\Re^2$, and let $\F$ be a family of pseudo-discs. Define a graph $G'$ whose vertices correspond to the sets in $B$, where two sets $b,b' \in B$ are connected by an edge if there exists a set $d \in \F$ such that the only elements of $B$ that $d$ intersects are $b,b'$ (i.e., $\{b'' \in B: b'' \cap d \neq \emptyset\} = \{b,b'\}$). Then $G'$ is planar.
\end{lemma}
For any $H' \subset \hat{H}$, the Delauney graph $D_{H'}$ satisfies the hypotheses of Lemma~\ref{Lemma:Rom2} (with $V(H')$ as the family of disjoint sets and $\widehat{V \setminus B}$ as the family of pseudo-discs). Therefore, by Lemma~\ref{Lemma:Rom2}, $D_{H'}$ is planar, which completes the proof.

\subsubsection{Coloring the vertices of $V \setminus B$}
\label{sec:VminusB}

Let $I=(V(I),E(I))$ be the hypergraph whose vertex set is $V(I)=V \setminus B$ and whose hyperedges are all sets of the form $e_b = \{d \in V \setminus B: d \cap b \neq \emptyset\}$, for all $b \in B$. Clearly, a CF-coloring of $I$ provides any $b \in B$ that has at least one neighbor in $V \setminus B$ with a uniquely-colored neighbor in $N_G(b)$, as desired. (For the remaining vertices in $B$, $N_G(b)$ is empty and so no uniquely-colored neighbor is needed.)

As the proof of CF colorability of $I$ is somewhat involved, we first present the proof in the special case where the $\F$ consists of discs. We then present the adjustments needed for the more general case of pseudo-discs. The arguments in the special case of discs already allows us to derive Theorem~\ref{Thm:Antennas} (see below). In the appendix we present an alternative proof of the CF colorability of $I$, using the results of ~\cite{k-admissible}. While it is somewhat simpler than the proof presented here, it uses two strong results as `black boxes' and does not yield the tools needed to prove Theorem~\ref{Thm:Antennas}, and thus, we prefer to delay it to the appendix.

\medskip

\noindent \textbf{The case of discs}

\medskip

By Theorem~\ref{Lemma:Proper-to-CF}, in order to show that $I$ admits a CF-coloring with $O(\log (n-p))$ colors, it is sufficient to show that there exists a constant $k$ such that any induced sub-hypergraph $I' \subset I$ adimts a proper coloring with $k$ colors. We will show this (for $k=6$) in two steps:

\medskip \noindent (1). For any induced sub-hypergraph $I' \subset I$, we define an auxiliary graph $G_{I'}$, and show that it is planar. So, by Euler's formula, it follows that it has a vertex of degree $\leq 5$.

\medskip \noindent (2). We use the low-degree vertices of the graphs $G_{I'}$ to define an inductive process that yields a proper coloring of $I$ and any of its induced sub-hypergraphs with 6 colors.

\medskip \noindent \textbf{Step 1. Studying the auxiliary graphs $G_{I'}$.} For any induced sub-hypergraph $I' \subset I$, we define $G_{I'}$ to be the graph whose vertices are $V(G_{I'})=V(I') \subset \F \setminus B$, and connect two vertices $d_1,d_2 \in V(G_{I'})$ by an edge if there exists $b \in B$ such that the only discs in $V(I')$ which $b$ intersects are $d_1,d_2$ (i.e., $\{d \in V(I'): d \cap b \neq \emptyset\} = \{d_1,d_2\}$).
\begin{claim}\label{Claim:G_Iplanar}
For any $I' \subset I$, the graph $G_{I'}$ is planar, and thus contains a vertex of degree $\leq 5$.
\end{claim}

\begin{proof}
We present the proof for $I'=I$. It will be clear from the proof that the same argument applies for any induced sub-hypergraph $I' \subset I$.
By the classical Hanani-Tutte theorem \cite{Tutte70} (see also
\cite{CH34}), a graph is planar if and only if it admits a plane drawing in which any two edges that do not share a vertex intersect an even number of times. We will present such a drawing of $G_I$.

For each disc $d$, we denote the center of $d$ by $Cent(d)$. For any $(d_1,d_2) \in E(G_I)$, let $b \in B$ denote a disc such that $\{d \in \F \setminus B: d \cap b \neq \emptyset\} = \{d_1,d_2\}$. Connect $Cent(d_1)$ by a straight line segment to $Cent(b)$ and connect $Cent(b)$ by a straight line segment to $Cent(d_2)$. The polygonal line that connects $Cent(d_1)$ with $Cent(d_2)$ (through $Cent(b)$), which we denote $e_{d_1,d_2}$, will be the drawing of the edge $(d_1,d_2)$.

By the Hanani-Tutte theorem, it is sufficient to prove that for any distinct $d_1,d_2,d'_1,d'_2$, the edges $e_{d_1,d_2}$ and $e_{d'_1,d'_2}$ do not intersect. Denote the discs of $B$ used in the drawing of the edges $e_{d_1,d_2}$ and $e_{d'_1,d'_2}$ by $b,b'$ (respectively). By symmetry, it is sufficient to show that the segment $[Cent(d_1),Cent(b)]$ does not intersect the segment $[Cent(d'_1),Cent(b')]$.

Let $x \in d_1 \cap b$ and $y \in d'_1 \cap b'$. Let $\ell$ be the perpendicular bisector of the segment $[x,y]$. We claim that the segments $[Cent(d_1),Cent(b)]$ and $[Cent(d'_1),Cent(b')]$ lie on different sides of $\ell$, and thus they cannot intersect. See Figure~\ref{fig:circles1} for an illustration.

\begin{figure}[tb]
\begin{center}
\scalebox{0.6}{
\includegraphics{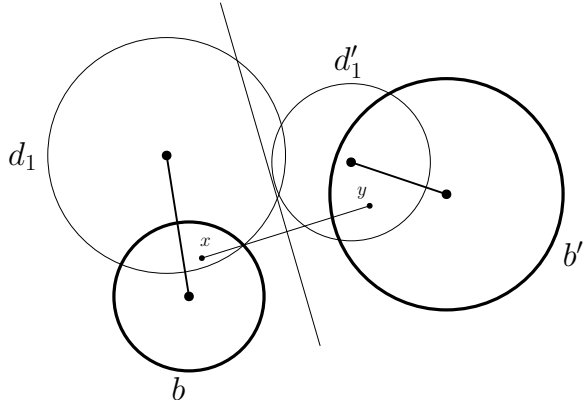}
} \caption{Illustration of the proof of Claim~\ref{Claim:G_Iplanar}} \label{fig:circles1}
\end{center}
\end{figure}

Observe that since $x \in d_1$ and $y \not \in d_1$, then $Cent(d_1)$ is closer to $x$ than to $y$.
Similarly, since $x \in b$ and $y \not \in b$, then $Cent(b)$ is closer to $x$ than to $y$.
  Hence, both $Cent(d_1)$ and $Cent(b)$ lie on the same side of $\ell$ as $x$, which implies that the entire segment $[Cent(d_1),Cent(b)]$ lies on the same side of $\ell$ as $x$. For the same reason, the entire segment $[Cent(d'_1),Cent(b')]$ lies on the same side of $\ell$ as $y$. Therefore, the segments do not intersect, which completes the proof.
\end{proof}

%

\medskip \noindent \textbf{Step 2. Coloring any sub-hypergraph of $I$ with 6 colors.} We show how to properly color $I$ with $6$ colors. It will be clear from the proof that the same argument applies to any induced sub-hypergraph of $I$.

We define a linear ordering on the vertices of $I$ as follows. By Claim~\ref{Claim:G_Iplanar}, $G_I$ contains a vertex of degree $\leq 5$. We choose (arbitrarily) one such vertex $v_1$, and put it as the \emph{last} in that ordering. Then we consider the induced sub-hypergraph $I_1$ of $I$ whose vertex set is $V(I) \setminus \{v_1\}$ and consider the corresponding auxiliary graph $G_{I_1}$. As before, by Claim~\ref{Claim:G_Iplanar} it has a vertex of degree $\leq 5$. We choose such a vertex $v_2$ and put in the second-to-last place of the ordering. We continue in this fashion until we obtain a full ordering $\{v_m,v_{m-1},\ldots,v_2,v_1\}$ of the vertices of $I$ (note that the vertex set of $G_I$ is equal to the vertex set of $I$).

Now, we define the coloring of $I$ inductively. The vertices are colored one-by-one according to the ordering. The first vertex ($v_m$) is given an arbitrary color. Assume that all vertices $v_m,v_{m-1},\ldots,v_{i+1}$ are already colored and consider the vertex $v_i$. Let $I_i$ be the induced sub-hypergraph of $I$ whose vertex set is $\{v_m,v_{m-1},\ldots,v_{i+1},v_i\}$, and let $G_{I_i}$ be the corresponding auxiliary graph. By construction, the degree of $v_i$ in $G_{I_i}$ is at most 5. Hence, we can assign to $v_i$ a color that is not used by any of its neighbors in $G_{I_i}$. We choose arbitrarily one of these `available' colors and assign it to $v_i$. We claim that the resulting coloring is proper.

Indeed, assume to the contrary that the coloring is not proper. Hence, there is a step $i$ such that the color assignment to $v_i$ created a monochromatic hyperedge $e_b$ with $|e_b| \geq 2$. Let $i$ be the earliest stage when such an event occurred. For any $e_b$ of size $\geq 3$, the edge $e_b \setminus \{v_i\}$ was not monochromatic by the minimality of $i$, and thus, $e_b$ also cannot be monochromatic. The same applies for any $e_b$ of size $2$ that does not contain $v_i$. The only remaining case is hyperedges $e_b$ of the form $e_b=\{v_i,v_j\}$ where $j>i$. For such an hyperedge, $(v_i,v_j)$ is an edge in the auxiliary graph $G_{I_i}$. Thus, by the coloring rule, the color of $v_i$ differs from the color of $v_j$, which means that $e_b$ is not monochromatic, a contradiction.

Therefore, we obtained a proper coloring of $I$ with 6 colors, and as mentioned above, a proper coloring of any induced sub-hypergraph of $I$ with 6 colors can be obtained in the same way. By Theorem~\ref{Lemma:Proper-to-CF}, this implies that $I$ admits a CF-coloring with $O(|V(I)|)=O(\log (n-p))$ colors, as asserted.
This completes the proof of Theorem~\ref{Thm:Pseudo-discs} in the easier case where the elements of $\F$ are \emph{discs} (rather than pseudo-discs).

\begin{remark}
We note that, with some extra effort involving additional geometric ideas, one can construct a proper coloring of any induced sub-hypergraph of $I$ with only 4 colors, which is obtained directly from the proper coloring of $G_I$. We do not present this improvement as its effect on the bound for CF-coloring is asymptomatically insignificant.
\end{remark}

\noindent \textbf{ Proof in the general case}

\medskip

We return to the general case where $\F$ is a family of pseudo-discs. The general strategy of the proof is similar to the case of discs.

By Theorem~\ref{Lemma:Proper-to-CF}, in order to show that $I$ admits a CF-coloring with $O(\log (n-p))$ colors, it is sufficient to show that there exists a constant $k$ such that any induced sub-hypergraph $I' \subset I$ admits a proper coloring with $k$ colors. We will show this (this time, for $k=7$) in three steps (where the latter two steps are similar to the case of discs, though with a somewhat more involved proof).

\medskip \noindent (1). We prune the family $\F \setminus B$ by removing some of the pseudo-discs such that each of the remaining pseudo-discs contains a point of depth 1 (i.e., a point that is contained in a single pseudo-disc). (All removed pseudo-discs will be colored with a single additional color.) We denote the resulting family by $\bar{\F}$ and the corresponding hypergraph by $\bar{I}$. Note that each of the removed pseudo-discs is covered by the union of the remaining ones in $\bar{\F}$.

\medskip \noindent (2). For any induced sub-hypergraph $I' \subset \bar{I}$, we define an auxiliary graph $G_{I'}$ like in the case of discs, show that it is planar (which is the most complex part of the argument), and deduce that it has a vertex of degree $\leq 5$.

\medskip \noindent (3). We use the low-degree vertices of the graphs $G_{I'}$ to define an inductive process that yields a proper coloring of $I$ and any of its induced sub-hypergraphs with 7 colors (where 6 colors are needed for $\bar{I}$ and an additional color is needed for the pseudo-discs removed at the beginning).

\medskip \noindent \textbf{Step 1. Pruning  $\F \setminus B$.} First, we order the elements of $\F \setminus B$ in an arbitrary ordering $d_1,d_2,\ldots,d_m$ (where $m=n-p=|\F \setminus B|$). Then, we go over the elements, denote $\F_0=\F \setminus B$, and define a sequence of subfamilies $\F_1,\F_2,\ldots \subset \F \setminus B$ as follows. At the $i$'th step, we check whether $d_i$ contains a point that is of depth 1 in the family $\F_{i-1}$ (i.e., whether there exists $x \in d_i$ such that $x \not \in d'$ for all $d \neq d' \in \F_{i-1}$). If it does not contain such point, we define $\F_{i}=\F_{i-1} \setminus \{d_i\}$ (that is, we remove the disc $d_i$ from the family). Otherwise, we define $\F_i=\F_{i-1}$.

At the end of the process, we remain with a subfamily $\bar{\F} \subset \F \setminus B$ that satisfies the following two properties:

\medskip \noindent (1). Each $d \in \bar{\F}$ contains a point of depth 1 (with respect to $\bar{\F}$).

\medskip \noindent (2).We have $\bigcup_{\{d \in \bar{\F}\}} d = \bigcup_{\{d' \in \F \setminus B\}} d'$.
\medskip We define $\bar{I}=(V,E)$ as the hypergraph whose vertex set is $V= \bar{F}$ and whose hyperedges are all sets of the form $e_b = \{d \in \bar{F}: d \cap b \neq \emptyset\}$, for all $b \in B$.

The `removed' pseudo-discs will be colored at a later stage.

\medskip \noindent \textbf{Step 2. Studying the auxiliary graphs $G_{I'}$ for $I' \subset \bar{I}$.}
Like in the case of discs, for any induced sub-hypergraph $I' \subset \bar{I}$, we define $G_{I'}$ to be the graph whose vertices are $V(G_{I'})=V(I') \subset \bar{\F}$, and connect two vertices $d_1,d_2 \in V(G_{I'})$ by an edge if there exists $b \in B$ such that the only discs in $V(I')$ which $b$ intersects are $d_1,d_2$ (i.e., $\{d \in V(I'): d \cap b \neq \emptyset\} = \{d_1,d_2\}$).
\begin{claim}\label{Claim:G_Iplanar2}
For any $I' \subset \bar{I}$, the graph $G_{I'}$ is planar, and thus contains a vertex of degree $\leq 5$.
\end{claim}

\begin{proof}
We present the proof for $I'=\bar{I}$. It will be clear from the proof that the same argument applies for any induced sub-hypergraph $I' \subset \bar{I}$. Also, since any planar graph has a vertex of degree $\leq 5$ as mentioned before, we will be done by proving that $G_{\bar{I}}$ is planar. Moreover, by the Hanani-Tutte theorem it will be sufficient to present a plane drawing of $G_{\bar{I}}$ in which any two edges that do not share a vertex intersect an even number of times. We will present now such a drawing of $G_{\bar{I}}$. A crucial fact we use here is that for any $d',d'' \subset \bar{\F}$, both $d' \setminus d''$ and $d' \cap d''$ are connected, since $\bar{\F}$ is a family of pseudo-discs.

For each pseudo-disc $d$, we choose a point of depth 1 in $d$ and call it $x_d$. (Note that by the construction of $\bar{I}$, any $d \in \bar{F}$ contains a point of depth 1.) For any $(d',d'') \in E(G_{\bar{I}})$, let $b \in B$ be a pseudo-disc such that $\{d \in \bar{\F}: d \cap b \neq \emptyset\} = \{d',d''\}$. We define a simple Jordan curve from $x_{d'}$ to $x_{d''}$ via $b$ that will be the drawing of the edge $(d',d'')$ in the following way:
\begin{itemize}
\item \textbf{Case 1: $d' \cap \partial b \not \subset d''$.} Since $d' \setminus d''$ is connected and intersects $b$, there exists $y' \in (d' \setminus d'') \cap \partial b$ and a simple Jordan curve $p_1$ in $d' \setminus d''$ from $x_{d'}$ to $y'$ that intersects $b$ only in $y'$. Likewise, there exists a simple Jordan curve $p_3$ in $d''$ from $x_{d''}$ to some $y'' \in d'' \cap \partial b$ that intersects $b$ only in $y''$. Furthermore, there exists a Jordan curve $p_2$ in $\partial b$ from $y'$ to $y''$. The drawing $e_{d',d''}$ of the edge $(d',d'')$ will be $p_1 \cup p_2 \cup p_3$. (Note that, by construction, this curve is simple.)

\item \textbf{Case 2: $d'' \cap \partial b \not \subset d'$.} This case is handled like the previous case, with the roles of $d',d''$ interchanged.

\item \textbf{Case 3: $d' \cap \partial b = d'' \cap \partial b$.} Pick $y' \in d' \cap d'' \cap \partial b$. Let $\hat{x_{d'}}$ ($\hat{x_{d''}}$) be the point in $d' \cap d''$ which is the closest to $x_{d'}$ (respectively, $x_{d''}$). Let $p_1$ be a simple curve in $d' \setminus d''$ from $x_{d'}$ to $\hat{x_{d'}}$, and let $p_3$ be a simple curve in $d'' \setminus d'$ from $x_{d''}$ to $\hat{x_{d''}}$. Since $\hat{x_{d'}},\hat{x_{d''}}$, and $y'$ belong to the same connected set $d' \cap d''$, we can connect $\hat{x_{d'}}$ to $\hat{x_{d''}}$ via $y'$ by a simple Jordan curve $p_2$. The drawing $e_{d',d''}$ of the edge $(d_1,d_2)$ will be $p_1 \cup p_2 \cup p_3$. This curve is indeed simple since each of the curves $p_1,p_2,p_3$ is simple and they belong to disjoint sets ($d' \setminus d''$, $d' \cap d''$, and $d'' \setminus d'$, respectively).
\end{itemize}

We claim that for any distinct $d',d'',f',f''$, the curves $e_{d',d''}=p_1 \cup p_2 \cup p_3$ and $e_{f',f''}= p'_1 \cup p'_2 \cup p'_3$ intersect an even number of times (and thus, by the Hanani-Tutte theorem, the graph $G_{\bar{I}}$ is planar). We will use the following lemma proved in~\cite{BPR13}.
\begin{lemma}[~\cite{BPR13}]\label{Lemma:BP}
Let $d',f'$ be two pseudo-discs in the plane. Let $x,y \in d' \setminus f'$ and let $a,b \in f' \setminus d'$. Let $e_1$ be a Jordan arc connecting $x$ and $y$ that is fully contained in $d'$, and let $e_2$ be Jordan arc connecting $a$ and $b$ that is fully contained in $f'$. Then $e_1$ and $e_2$ cross an even number of times.
\end{lemma}

Denote the pseudo-discs of $B$ used in the drawing of the edges $e_{d',d''}$ and $e_{f',f''}$ by $b_1,b_2$ (respectively).

The parts of $e_{d',d''}$ and $e_{f',f''}$ that lie on $\partial b_1$ and $\partial b_2$, respectively, do not intersect since $b_1 \cap b_2 = \emptyset$. (Note that unlike the case of discs, here we do use the fact that the elements of $B$ are mutually disjoint). Hence, without loss of generality it is sufficient to prove that the part of $e_{d',d''}$ from the point $x_{d'}$ until the first intersection with $\partial b_1$ (in Case 1, this is exactly $p_1$; in Case 3, this is $p_1$ and the first part of $p_2$) intersects the part of $e_{f',f''}$ from the point $x_{f'}$ until the first intersection with $\partial b_2$ an even number of times. Denote these parts of $e_{d',d''}$ and $e_{f',f''}$ by $e_1, e'_1$, respectively. Since $f' \cap b_1 = d' \cap b_2 = \emptyset$, and since the points $x_{d'},x_{f'}$ are of depth 1, we get that both endpoints of $e_1$ belong to $d' \setminus f'$, while both endpoints of $e'_1$ belong to $f' \setminus d'$. Therefore, by Lemma~\ref{Lemma:BP}, $e_1$ and $e'_1$ intersect an even number of times. This completes the proof.
\end{proof}

\medskip \noindent \textbf{Step 3. Coloring any sub-hypergraph of $\bar{I}$ with 7 colors.} This step is similar to the last step of the proof in the case of discs. The only difference is that all pseudo-discs that were removed during the `pruning' step are colored in an additional (7th) color. The proof presented in the case of discs applies verbatim to show that the coloring does not contain a monochromatic hyperedge in any of the first six colors. To complete the proof we note that there is no monochromatic edge colored 7, as existence of such an edge would contradict condition~2 that $\bar{F}$ satisfies at the end of the `pruning' step.

Thus, we proved that $I$, as well as any of its induced subgraphs, admits a proper coloring with 7 colors. By Theorem~\ref{Lemma:Proper-to-CF}, this implies that $I$ admits a CF-coloring with $O(\log |V(I)|) = O(\log (n-p))$ colors. Together with the coloring of $B$ with $O(\log p)$ additional colors presented in the previous section, we obtain $\chi_{CF}^{pn}(G) = O(\log n)$.

\subsection{Proof of Theorems~\ref{Thm:Antennas} and~\ref{Thm:List}}

First we present the proof of Theorem~\ref{Thm:Antennas}. Recall that in the theorem, $\F$ is a family of $n$ discs in the plane. We have to show that $\F$ can be colored by $O(\log n)$ colors in such a way that for any disc $b$, not necessarily in $\F$, the set $S_b = \{D \in \F: D \cap b \neq \emptyset\}$ contains a uniquely-colored element. Moreover, the bound $O(\log n)$ is asymptotically tight.

\medskip

We define a hypergraph $I$ by setting $V(I)=\F$ and taking the hyperedges to be all sets of the form $e_b = \{d \in \F: d \cap b \neq \emptyset\}$, for all discs $b$ (not necessarily in $\F$) in the plane. It is clear that a CF-coloring of $I$ is a coloring of $\F$ that satisfies the assertion of the theorem.

For any $I' \subset I$, we define a graph $G_{I'}$ whose vertices are $V(G_{I'})=V(I')$, and whose edges are all pairs $\{d_1,d_2\}$ such that there exists a disc $b$ for which $\{d \in V(I'): d \cap b \neq \emptyset\}=\{d_1,d_2\}$.

The proof of Claim~\ref{Claim:G_Iplanar} applies verbatim to show that for any $I' \subset I$, the graph $G_{I'}$ is planar. (Note that while in Claim~\ref{Claim:G_Iplanar} it is assumed that the discs of $B$ are pairwise disjoint, the proof does not use this assumption. The assumption is used only in the more complex case of pseudo-discs handled in Claim~\ref{Claim:G_Iplanar2}.) As shown above, this implies that $I$, as well as any of its induced sub-hypergraphs, admits a proper coloring with $6$ colors. Therefore, by Theorem~\ref{Lemma:Proper-to-CF}, $\chi_{CF}(I)=O(\log n)$, as asserted.
The tightness of the assertion follows from the tightness of Theorem~\ref{thm:ELRS1} showed in~\cite{ELRS} (as Theorem~\ref{thm:ELRS1} corresponds to a special case of Theorem~\ref{Thm:Antennas}).

\medskip


In order to present the proof of Theorem~\ref{Thm:List}, we need a few definitions.

A hypergraph $H=(V,E)$ with $V =\{v_1,v_2,\ldots,v_n\}$ is $k$-CF-choosable if for any family of $n$ lists $L_1,L_2,\ldots,L_n$ such that $|L_i| \geq k$ for all $i$, there exists a CF-coloring $\varphi$ of $H$ such that $\varphi(v_i) \in L_i$ for all $1 \leq i \leq n$. Such coloring is called a \emph{list coloring}. $\mathrm{ch}_{CF}(H)$ is the least integer $k$ such that $H$ is $k$-CF-choosable. Note that we always have $\chi_{CF}(H) \leq \mathrm{ch}_{CF}(H)$ since a CF-coloring of $H$ with $k$ colors corresponds to the special case of $k$-choosability in which $L_1=L_2=\ldots=L_n=\{1,2,\ldots,k\}$.

Theorem~\ref{Thm:List} follows immediately from the above argument, via the following theorem of Cheilaris et al.~\cite{CSS11}.
\begin{theorem}\label{Thm:CSS}
Let $H$ be a hypergraph. If any induced sub-hypergraph $H' \subset H$ admits a proper coloring with a constant number $k$ of colors, then $\mathrm{ch}_{CF}(H) \leq 1+ \log_{1+\frac{1}{k-1}} n = O(\log n)$.
\end{theorem}
Define the hypergraph $I$ as above. Since $I$, as well as any of its induced sub-hypergraphs, admits a proper coloring with 6 colors, Theorem~\ref{Thm:CSS} implies $\mathrm{ch}_{CF}(I) = O(\log n)$. The tightness of the assertion follows from the tightness of Theorem~\ref{Thm:Antennas}, as for any $H$,
$\mathrm{ch}_{CF}(I) \geq \chi_{CF}(H)$.

\appendix

\section{Intersection Graph of Axis-Parallel Rectangles}
\label{sec:rectangles}

Here we prove Theorem~\ref{Thm:Rectangles}, namely, that the intersection graph of any family of $n$ axis-parallel rectangles in the plane admits a closed CF-coloring with $O(\log n)$ colors.
The proof consists of two steps. First, we prove a lemma stating that the intersection graph of any family of intervals on a line admits a closed CF-coloring with only 3 colors. Then, we apply this lemma to the case of axis-parallel rectangles in the plane, using the observation that for any family $S$ of axis-parallel rectangles all of which intersect an axis-parallel line $\ell$, the intersection graph of $S$ is isomorphic to the intersection graph of the family of intervals $\{X \cap \ell: X \in S\}$ on the line $\ell$.

\begin{lemma}\label{Lemma:Line}
Let $\F$ be a finite family of of closed intervals on a line $\ell$ and let $G$ be the intersection graph of $\F$. Then $G$ admits a closed CF-coloring with 3 colors.
\end{lemma}

\begin{proof}
Let $\bigcup \F = \cup_{s \in \F} s$ be the union of all intervals of $\F$ . For each interval $s \in \F$, denote its left and right endpoints by $l(s),r(s)$, respectively, so that $s=[l(s),r(s)]$.
We provide a closed CF-coloring of $G$ in an inductive fashion. Specifically, we define an inductive process in which we color some intervals $s_1,s_2,s_3,\ldots$ with colors $1,2$ alternately, where $\{s_i\}$ are defined as follows.
Let $S_1 = \{s \in \F: l(s) = \min_{t \in \F} l(t)\}$ be the sub-family of all intervals in $\F$ whose left endpoint is the leftmost. Let $s_1$ the interval in $S_1$ whose right endpoint is maximal. Color $s_1$ with color 1.
Now, assume that we performed $i$ coloring steps and that the interval $s_i$ is colored 1 (respectively, 2). Let
\begin{align*}
S_{i+1} = \{s \in \F: &l(s) \leq r(s_i) \wedge r(s)>r(s_i)\} \cup \\
&\{s \in \F: (l(s) > r(s_i)) \wedge \left([r(s_i),l(s)) \cap \bigcup \F = \emptyset \right) \},
\end{align*}
that is, the set of intervals whose right endpoints is greater than $r(s_i)$, and whose left endpoint does not `leave behind it' any point $x \in \bigcup \F$ that was not covered by the previous $s_j$'s. Let $s_{i+1}$ be the interval in $S_{i+1}$ whose right endpoint is maximal. Color $s_{i+1}$ with color 2 (respectively, 1). Note that by the construction, the sequence $\{r(s_i)\}$ is strictly increasing and that $\cup_i s_i = \bigcup \F$.

After completing the inductive process, we color all remaining elements of $\F$ with color 3. We claim that the resulting coloring is indeed a closed CF-coloring, since each $s \in \F$ has a unique `neighbor' colored either 1 or 2 (where if $s$ is colored 1 or 2, it is the uniquely colored neighbor of itself). By symmetry, it is clearly sufficient to prove that the following three statements hold:

\medskip \noindent (1). Any two intervals $s_i,s_j$ such that $j \geq i+2$ are disjoint, and in particular, any $s_i,s_j$ that have the same color are disjoint.

\medskip \noindent (2). Any $s \in \F$ that intersects both $s_i$ and $s_{i+2}$ for some $i$, intersects $s_{i+1}$.

\medskip \noindent (3). There does not exist $s \in \F$ colored 3 such that $N_G(s)$ contains at least two intervals colored 1 and at least two intervals colored 2.

Statement (1) implies that any interval $s \in \F$ colored 1 or 2 is itself's unique colored neighbor. Statements (2) and (3), together with the observation that each $s \in \F$ intersects some $s_i$ (since $\cup_i s_i = \bigcup \F$), imply that any interval $s \in \F$ colored 3 has a unique neighbor colored either 1 or 2,  which will complete the proof.

\medskip \noindent \textbf{Proof of (1).} Assume that $i<j$ and $s_i \cap s_j \neq \emptyset$. As the sequence $\{r(s_i)\}$ is strictly increasing, we have $r(s_j)>r(s_i)$. On the other hand, since $s_i \cap s_j \neq \emptyset$, we have $l(s_j) \leq r(s_i)$. Therefore, by the definition of $S_{i+1}$, we have $s_j \in S_{i+1}$. Thus, by the choice of $s_{i+1}$, $r(s_j) \leq r(s_{i+1})$. Using again the fact that $\{r(s_i)\}$ is strictly increasing, we must have $j=i+1$, as asserted. The fact that any $s_i,s_j$ that have the same color satisfy $j \geq i+2$ or $i \geq j+2$ holds since the $s_i$'s are colored 1,2 alternately.

\medskip \noindent \textbf{Proof of (2).} Let $s \in \F$ be an interval that intersects $s_i$ and $s_{i+2}$, for some $i$. By Statement (1), we have $l(s) \leq r(s_i)$ and $r(s) \geq l(s_{i+2}) > r(s_i)$. By the definition of the set $S_{i+1}$, this implies that all elements $t \in S_{i+1}$ satisfy $l(t) \leq r(s_i)$ (since if for some $t$, $l(t)>r(s_i)$ then $[r(s_i),l(t)) \cap \bigcup \F \neq \emptyset$ because of $s$). In particular, we have $l(s_{i+1}) \leq r(s_i)$ and $r(s_{i+1})>r(s_i)$. Therefore, $r(s_i) \in s \cap s_{i+1}$, and thus $s$ intersects $s_i$, as asserted.

\medskip \noindent \textbf{Proof of (3).} Let $s \in \F$ be an interval colored 3 that intersects at least two intervals colored 1 and at least two intervals colored 2. By Statement (1), this implies that there exists $i$ such that $s$ intersects $s_i,s_{i+1},s_{i+2},$ and $s_{i+3}$. (Note that if some $t \in \F$ intersects $s_i$ and $s_{i+2k}$ for some $k$, then $t$ intersects all intervals $s_{i+2j}$, $1 \leq j \leq k-1$, as the intervals $\{s_{i+2j}\}$ are pairwise disjoint and the sequence of their right endpoints is increasing.) This, in turn, implies that $l(s)<r(s_i)$ and $r(s)>l(s_{i+3})>r(s_{i+1})>r(s_i)$. Hence, $s \in S_i$ and its right endpoint is larger than $r(s_{i+1})$ which contradicts the choice of $s_{i+1}$.
This completes the proof of the Lemma.
\end{proof}

\noindent Now we are ready to present the proof of Theorem~\ref{Thm:Rectangles}.
\begin{proof}[Proof of Theorem~\ref{Thm:Rectangles}]
Denote the minimum number of colors required for closed CF-coloring the intersection graph of \emph{any} family of $n$ axis-parallel rectangles by $f(n)$. We obtain a recursive bound on $f(n)$.

Let $\F$ be a family of $n$ axis-parallel rectangles in the plane, and assume, for sake of convenience, that $n$ is a power of 2. Let $\ell$ be a vertical line such that at most $n/2$ elements of $\F$ lie strictly to the left of $\ell$ and at most $n/2$ elements of $\F$ lie strictly to the right of $\ell$. Let $\F'$ (respectively, $\F''$) be the sub-family of all elements of $\F$ that lie strictly to the right (respectively, left) of $\ell$, and let $\F_{\ell}$ denote the sub-family $\{S \in \F: S \cap \ell \neq \emptyset\}$. By Lemma~\ref{Lemma:Line}, the intersection graph of $\F_{\ell}$ admits a closed CF-coloring with 3 colors. (This holds since the intersection graph of $F_{\ell}$ is easily seen to be isomorphic to the intersection graph of the family of intervals $\{S \cap \ell: S \in F_{\ell}\}$ on the line $\ell$). By induction, the intersection graphs of each of the families $\F', \F''$ admit a closed CF-coloring with $f(n/2)$ colors. We use the same colors for $\F', \F''$ , and three additional colors for $\F_{\ell}$. Thus, we obtain a coloring of $\F$ with $f(n/2)+3$ colors. We claim that this is a closed CF-coloring.
Indeed, for each $s \in \F$ there exists a step in the recursive process in which $s$ belongs to the sub-family $\F_{\ell'}$ for some $\ell'$. The coloring of $\F_{\ell'}$ provides $s$ with a uniquely colored neighbor. (Note that while the colors used in the coloring of $\F_{\ell'}$  are used also for other elements of $\F$, they are used only for elements that do not intersect elements of $\F_{\ell'}$ , since they are separated from $\F_{\ell'}$  by a line.)
Therefore, we have $f(n) \leq f(n/2)+3$, which implies $f(n)=O(\log n)$, as asserted.
\end{proof}

We note that with a little more effort, one can prove that any family of closed intervals on a line admits a pointed CF-coloring with 4 colors. However, this coloring cannot be leveraged into a pointed CF-coloring of the intersection graph of axis-parallel rectangles in the plane (like we did in the closed CF-coloring case in the proof of Theorem~\ref{Thm:Rectangles}). Indeed, it may occur that some $s \in \F_{\ell}$ has no neighbors in $\F_{\ell}$ but does have neighbors in, say, $\F'$. For those neighbors, we have no promise that one of them is uniquely colored.


\section{Intersection Graph of $\rho$-fat Objects}
\label{sec:rho-fat}

Here we prove Proposition~\ref{Thm:rho-fat} which asserts an upper bound on the closed and pointed CF chromatic numbers of the intersection graph of $\rho$-fat objects in the plane.
\begin{definition}
A simple Jordan region $C$ in the plane is called $\rho$-fat if there exists $x \in C$ and discs $A,B$ centered at $x$ with radiuses $r_A,r_B$ (respectively), such that $A \subset C \subset B$ and $r_B/r_A \leq \rho$.
\end{definition}
These objects were studied in numerous papers (see, e.g., Section~2 of the survey~\cite{APS-union} and the references therein).

For a $\rho$-fat object $C$, we say that the \emph{size} of $C$ is the maximal possible radius, $\max \{r_A\}$, where the maximum is taken over all pairs of concentric discs $A,B$ with $A \subset C \subset B$ and $r_B/r_A \leq \rho$. The \emph{size-ratio} of a family $\F$ of $\rho$-fat objects is $\max_{C_1,C_2 \in \F} \left(\frac{\mathrm{size}(C_1)}{\mathrm{size}(C_2)} \right)$.

We recall the formulation of Proposition~\ref{Thm:rho-fat}.
\begin{proposition*}
Let $k, \rho \geq 1$, and let $\F$ be a finite family of $\rho$-fat objects in the plane of size-ratio $k$. Let $G$ be the intersection graph of $\F$. Then:
\begin{enumerate}
\item $\chi_{CF}^{pn}(G) = O(k^2 \rho^2)$.
\item $\chi_{CF}^{cn}(G) = O(\log k \cdot \rho^2)$.
\end{enumerate}
\end{proposition*}

\begin{proof}
The proof uses a packing argument.

For each $C \in \F$, let $x_C \in C$ be a point for which there exist discs $A,B$ centered at $x_C$ with radiuses $r_A,r_B$ (respectively), such that $A \subset C \subset B$, $r_B / r_A \leq \rho$, and $\mathrm{size}(C)=r_A$. (Such a point must exist by the definition of $\mathrm{size}(C))$. We assume w.l.o.g. that $\min_{C \in \F} \mathrm{size}(C) =1$. We slightly abuse the notation and say that $C,C' \in F$ are neighbors if the corresponding vertices of $G$ are neighbors, i.e., if $C \cap C' \neq \emptyset$.

We start with the proof of Part~(1). Consider the $\ell^1$-grid in the plane (i.e., the standard division of the plane to $1 \times 1$ squares). Since $\F$ is finite, we can assume w.l.o.g. that all points $x_C$, $C \in \F$ lie in the interiors of cells of the grid. We cover the plane by squares of size $(4k \lceil \rho \rceil +1)  \times (4k \lceil \rho \rceil +1)$ whose sides lie on the grid. Let $t= (4k \lceil \rho \rceil +1)  \times (4k \lceil \rho \rceil +1)$ be the number of cells in each square. We take an arbitrary square, and assign to each cell $A$ in the square a color $c(A) \in \{1,2,\ldots,t\}$, such that each color is assigned exactly once. Then, we copy the coloring to all other squares cyclically, such that two cells get the same color if and only if their $x,y$ coordinates are equal modulo $4k \lceil \rho \rceil +1$.

We define a coloring of $\F$ with the color set $\{(i,\ell):1 \leq i \leq t+1, \ell=1,2\}$ in a two-phase procedure:
\begin{enumerate}
\item \textbf{Initial coloring:} For each cell $A$ that contains points of the form $x_C$, we choose one of such $C$'s arbitrarily and assign to it the color $(c(A),1)$. Then we assign the color $(t+1,1)$ to all remaining elements of $\F$.

\item \textbf{Recoloring:} For each $C \in \F$ colored $(i,1)$ with $i \leq t$ whose set of neighbors is non-empty, we check whether $C$ has a neighbor colored $(i',1)$ for some $i' \leq t$. If not, we choose arbitrarily a neighbor $C'$ of $C$ (colored $(t+1,1)$) and re-color it $(i,2)$. All the remaining elements of $\F$ stay unchanged.
\end{enumerate}
We claim that the resulting coloring is a pointed CF coloring of $\F$ with $2t+1=O(k^2 \rho^2)$ colors. Let $C \in \F$ be an object whose neighbor set is non-empty. We consider three cases:
\begin{enumerate}
\item If $C$ is colored $(t+1,1)$ then there exists $C' \in \F$ colored $(i,1)$ for some $i \leq t$, such that $x_{C'}$ lies in the same cell as $x_C$. Since each of $C,C'$ contains a disc of radius $\geq 1$ centered as $x_C,x_{C'}$ (respectively), we have $C \cap C' \neq \emptyset$, and thus, $C$ has a neighbor colored $(i,1)$. On the other hand, $C$ can have at most one neighbor colored $(i,1)$ since $C$ is included in a disc of radius $\leq k \rho$ centered at $x_C$, and all other cells whose color is $i$ are `too far' from $x_C$. (Note that the size of the square was chosen in order to assure that this property holds.) Therefore, $C'$ is a uniquely colored neighbor of $C$.

\item If $C$ is colored $(i,1)$ for some $i \leq t$, then by the second phase of the construction, $C$ has either a neighbor colored $(i',1)$ for some $i' \leq t$, or a neighbor colored $(i,2)$. In either case, this is a uniquely colored neighbor by the argument of the previous case.

\item If $C$ is colored $(i,2)$ then by the second phase of the construction, it must have a neighbor colored $(i,1)$ and this is a uniquely colored neighbor.
\end{enumerate}
Thus, the coloring is indeed a pointed CF coloring of $\F$ with $O(k^2 \rho^2)$ colors, as asserted.

\medskip

To prove Part~(2) of the Proposition, we divide the family $\F$ into sub-families $\F_i=\{C \in \F: 2^{i-1} \leq \mathrm{size}(C) < 2^i\}$, for $i=1,2,\ldots,\lceil \log k \rceil$. Since each $\F_i$ is a family of $\rho$-fat regions with size-ratio $\leq 2$, it follows from Part~(1) of the Proposition that $\F_i$ admits a pointed CF coloring with $O(\rho^2)$ colors. Furthermore, as mentioned in the introduction, this implies that $\F_i$ admits a \emph{closed} CF coloring with $O(\rho^2)$ colors.

We define a coloring of $\F$ with $O(\log k \cdot \rho^2)$ colors by taking a closed CF coloring of each $\F_i$ with $O(\rho^2)$ \emph{fresh} colors and taking the union of these colorings as the coloring of $\F$. We claim that the resulting coloring is a closed CF coloring of $\F$.

Indeed, let $C \in \F$ and assume w.l.o.g. that $C \in \F_j$. The closed CF coloring of $\F_j$ supplies $C$ with a uniquely colored neighbor with respect to $\F_j$ (which may be $C$ itself since we consider a closed CF coloring). This neighbor is uniquely colored also with respect to $\F$ since the colors used in the coloring of $\F_j$ are not used for any other element of $\F$. This completes the proof.
\end{proof}

We note that the argument of Part~(2) does not yield a pointed CF coloring of $\F$ with $\log k \cdot \rho^2$ colors. Indeed, if some $C \in \F_j$ has no neighbors in $\F_j$ but does have neighbors in $\F_{\ell}$ for some $\ell \neq j$, then we have no promise that $C$ has a uniquely colored neighbor (as all neighbors of $C$ in $\F_{\ell}$ may have the same color).

\section{An alternative proof of Theorem~\ref{Thm:Pseudo-discs}}
\label{sec:alternative}

Here we sketch an alternative proof of  the second part of Theorem~\ref{Thm:Pseudo-discs} (which appears in section~\ref{sec:VminusB}).
Recall that in this context $\F$ is a family of pseudo-discs in the plane, $G=(V,E)$ is the intersection graph of $\F$, and $B$ is a maximal independent set in $V$.
 We defined $I$ as the hypergraph whose vertex set is $V(I)=V \setminus B$ and whose hyperedges are all sets of the form $e_b = \{d \in V \setminus B: d \cap b \neq \emptyset\}$, for all $b \in B$.
The main step in section~\ref{sec:VminusB} is a proper coloring of any sub-hypergraph $I' \subset I$ with 7 colors. Here we present a shorter proof of a slightly weaker statement:
\begin{proposition}\label{pro:appendix}
Any induced sub-hypergraph $I' \subset I$ admits a proper coloring with a constant number of colors.
\end{proposition}

We start with a few definitions and results needed in the sequel:
\begin{definition}

 Let $\C$ be a family of simply connected regions bounded by simple closed curves in general position in the plane. Let $k$ be an even integer. $C$ is called $\emph{k-admissible}$ if for any pair $C_1,C_2 \in \C$

\begin{enumerate}
\item $C_i \setminus C_j$ and $C_j \setminus C_i$ are connected, and
\item $\partial C_i$ and $\partial C_j$ cross in at most $k$ points.
\end{enumerate}
\end{definition}

\begin{definition}
For a finite family $\mathcal{C}=\{C_1,C_2,\ldots,C_n\}$ of geometric objects in the plane, the \emph{union complexity} of $\mathcal{C}$ is the number of faces of all dimensions of the arrangement of the objects' boundaries, which lie on the boundary of $ \cup_{i=1}^n C_i$.
\end{definition}

\begin{theorem}[\cite{k-admissible}]\label{thm:k_admiss}
Let $\C=\{C_1,C_2,\ldots,C_n\}$ be a k-admissible family of $n \geq 3$ simply connected regions in general position in the plane. Then the union complexity of $\C$ is at most $O(kn)$.
\end{theorem}

The following special case of theorem~\ref{thm:k_admiss} was actually proved earlier by Kedem et al.:
\begin{corollary}[\cite{KLPS}]\label{cor:pseudo_d}
Any finite family of pseudo-discs in the plane has a linear union complexity.
\end{corollary}

\begin{theorem}[\cite{smoro}]\label{thm:linear_comp}
Let $\F$ be a family of $n$ simple Jordan regions in the plane with a linear union complexity, and let I be the hypergraph defined by $V(I)=\{C: C \in \F\}$ and $E(I)=\{ \{C \in \F : p \in \C \} : p \in \Re^2 \}$. Then there exists a constant number $\ell$ (that depends on the union complexity) such that $\chi(I) \leq \ell$.
\end{theorem}

Note that if for some $b \in B$, there exists a point $p \in b$ which is contained in at least two pseudo-discs in $V \setminus B$, then in the proper coloring of $I$ given by Theorem~\ref{thm:linear_comp} (using Corollary~\ref{cor:pseudo_d}), $e_b$ is not monochromatic.
The idea here, is to consider $b \in B$ with $|e_b| >1$ that does not contain such a point $p$, and to create such a point `fictively': We choose a pair $d_1,d_2 \in e_b$ and perform on them a local `mutation', inside the interior of $b$, that makes them cross each other.
These mutations will be performed in such a way that the family obtained from $V \setminus B$ by them, will be 4-admissible. Using Theorems~\ref{thm:k_admiss} and~\ref{thm:linear_comp}, this will provide us with a proper coloring of $I$ with a constant number of colors.

\begin{proof}[Proof of Proposition~\ref{pro:appendix}.]
We present the proof for $I'=I$. It will be clear from the proof that the same argument applies for any induced sub-hypergraph $I' \subset I$.
We arrange the elements of $B$ in an arbitrary order $\{b_1,\ldots,b_p\}$, and go over them according to this order. For each $i$, we consider the hyperedge $e_{b_i}$ and make a `mutation' in one pair $(d_1,d_2) \in e_{b_i}$ if and only if the following two conditions hold:
\begin{enumerate}
\item The family $\{d \cap  b : d \in e_{b_i} \}$ contains at least two elements, and all of its elements are pairwise disjoint.
\item The set $\{d : d \in e_{b_i} \}$ does not contain a pair $\{d_1,d_2\}$ for which we performed a mutation in a previous step.
\end {enumerate}

If in step $i$ these two conditions hold, we choose $d_1,d_2 \in e_{b_i}$ and make the mutation by adding to each of them a thin ``tail" inside $b_i$ such that the two ``tails" cross in exactly two points (see Figure~\ref{fig:tails}).  Note that since $\F$ is a family of pseudo-discs, $b_i'=b_i \setminus \cup_{ d \in {e_{b_i} \setminus \{d_1,d_2\}}} {( d \cap b_i )}$ is connected, and so there exists a path in $b_i'$ from a point in $d_1 \cap b_i$ to a point in $d_2 \cap b_i$. We can thicken this path a bit, such that it remains contained in $b'$, and split it into two ``tails" -- one that emanates from $d_1 \cap b_i$ and the other that emanates from $d_2\cap b_i$ -- that cross in exactly two points.

\begin{figure}[tb]
\begin{center}
\scalebox{0.8}{
\includegraphics{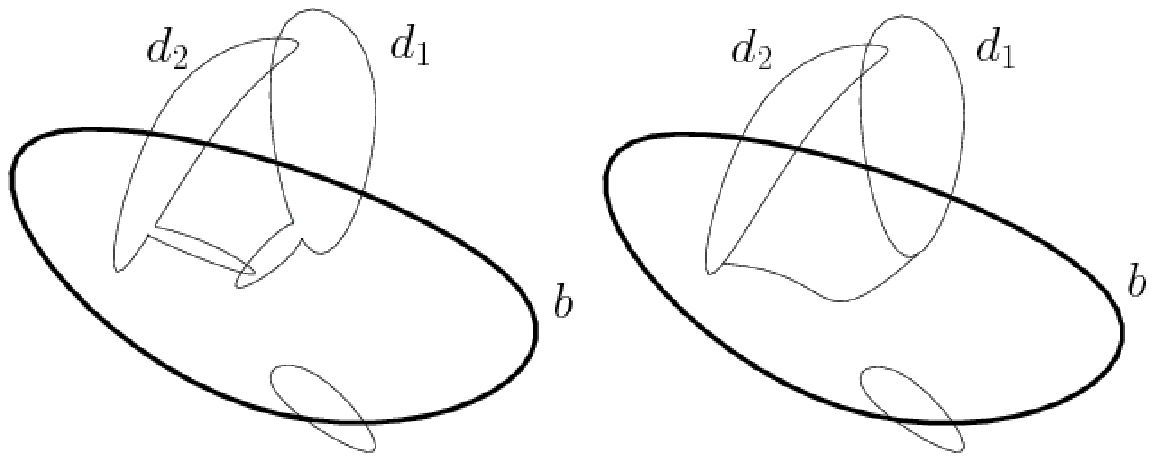}
} \caption{Illustration of the proof of Proposition~\ref{pro:appendix}. The figure demonstrates mutation performed to a pair
of pseudo-disks $d_1,d_2 \in e_{b}$. On the right of the figure, there is a path inside $b$ between $d_1$ and $d_2$. On the
left of the figure, the path is thickened into a pair of `tails' -- one that emanates from $d_1$ and another that emanates
from $d_2$.} \label{fig:tails}
\end{center}
\end{figure}

It is not difficult to see that due to the second condition, the sequence of mutations transforms the family $V \setminus B$ to a $4$-admissible family, which we call $\F'$. By Theorem~\ref{thm:k_admiss}, the union complexity of $\F'$ is linear. Theorem~\ref{thm:linear_comp} implies that the family $\F'$ admits a proper coloring $\alpha$ with a constant number of colors, such that for any point $p \in \Re^2$, if the family $\{d:d \in \F', p \in d\}$ contains more than one element, then its elements are colored with at least two different colors.

We claim that this coloring $\alpha$ of $\F'$, naturally induces a coloring of $V \setminus B$ which is a proper coloring of $I$. Indeed, considering any $b \in B$ with $|e_b|>1$, there are three cases:
\begin{enumerate}
\item There exists a point $p \in b$ covered by $r$ elements of $e_b$ for $r \geq 2$ (meaning that no mutation was done inside $b$). In this case, $\alpha$ supplies the elements of $\{d \in e_b : p \in d \}$ with at least two colors.
\item A mutation was done inside $b$ to a pair $(d_1,d_2) \in e_b$. In this case, after the mutation there exists a point $p \in b$ covered by exactly two elements of $e_b$: $d_1$ and $d_2$. Thus, $d_1,d_2$ are colored by different colors and so $e_b$ is not monochromatic. (Recall that the coloring $\alpha$ was performed after all the mutations).
\item No mutation was done inside $b$, since in the corresponding step, condition 2 was not satisfied due to a pair $(d_1,d_2) \in e_b$. In this case, the argument of case 2 implies that $d_1,d_2$ get different colors.
\end{enumerate}
Thus, $I$ admits a proper coloring with a constant number of colors, and since the same argument can be applied also to any sub-hypergraph of $I$, this completes the proof of Proposition~\ref{pro:appendix}.
\end{proof}

\section{Tightness of Theorem~\ref{Thm:Pseudo-discs}}
\label{App:Tightness}

Here we complete the proof of Theorem~\ref{Thm:Pseudo-discs} by presenting an explicit example of a graph $G$ that satisfies $\chi_{CF}^{pn}(G) = \Omega(\log n)$ (which shows that the assertion of the theorem is tight).

In~\cite{ELRS} it was shown that for any $n$ there exists a family $\F$ of $n$ unit discs in the plane such that at least $\Omega(\log n)$ colors are required to color the discs of $\F$ in such a way that for any point in the plane, the set of discs that contains it will contain a uniquely colored element. It is clear that the same assertion holds if we construct a set $S$ of points by taking one representative from each cell in the arrangement of $\F$, and weaken the requirement from `for any point' to `for any point in $S$'. It is well known and an easy fact (by Euler's formula) that $|S|$ (which equals to the number of cells in the arrangement of $\F$) is at most quadratic in $n$.

Now, replace each point $x \in S$ with a tiny disc $d_x$ inside the same cell of $x$ that is sufficiently small so that any $D \in \F$ intersects $d_x$ if and only if it intersects $x$. Consider the intersection graph $G$ of the family $\F \cup \{d_x:x \in S\}$ (i.e., all the elements of $\F$ and all the new small discs). In any coloring of $G$, the pointed neighborhood of a disc $d_x$ has a uniquely colored element if and only if the set $\{D \in \F: x \in D\}$ has a uniquely colored element. Hence, any pointed CF-coloring of $G$ provides a coloring of $\F$ such that for any $x \in S$, the set of discs that contains $x$ has a uniquely colored element. By the result of~\cite{ELRS} it follows that any such coloring uses at least $\Omega(\log n)$ colors.

We have thus obtained an intersection graph $G$ of a family of discs such that $|V(G)|=O(n^2)$, while $\chi_{CF}^{pn}(G) = \Omega(\log n) = \Omega(\log |V(G)|)$. This completes the proof of Theorem~\ref{Thm:Pseudo-discs}.

\bibliographystyle{alpha}
\bibliography{references}

\end{document}